\documentclass[11pt]{amsart}
\usepackage{amssymb}
\usepackage{amsmath}
\usepackage{amsfonts}
\usepackage{graphicx}

\usepackage[total={17cm,22cm},top=2.5cm, left=2.3cm]{geometry}
\usepackage{hyperref}
    \usepackage{aeguill}
    \usepackage{type1cm}

\theoremstyle{plain}
\newtheorem{thm}{Theorem}[section]

\newtheorem{claim}[thm]{Claim}

\newtheorem{corollary}[thm]{Corollary}

\newtheorem{definition}[thm]{Definition}

\newtheorem{lem}[thm]{Lemma}

\newtheorem{proposition}[thm]{Proposition}

\newtheorem{fact}[thm]{Fact}

\newtheorem{rem}[thm]{Remark}

\numberwithin{equation}{section}
\newcommand{\N}{\mathbb{N}}

\newcommand{\R}{\mathbb{R}}

\newcommand{\espan}{\operatorname{span}}
\newcommand{\Ker}{\operatorname{Ker}}
\newcommand{\rank}{\operatorname{rank}}

\usepackage[usenames,dvipsnames]{color}

\usepackage[dvipsnames]{xcolor}

\begin{document}

\title[Smooth extractions of critical points]{Extraction of critical points of smooth functions on Banach spaces}

\author{Miguel Garc\'ia-Bravo}
\address{ICMAT (CSIC-UAM-UC3-UCM), Calle Nicol\'as Cabrera 13-15.
28049 Madrid SPAIN}
\email{miguel.garcia@icmat.es}

\date{September 2019}


\thanks{The author was supported by Programa Internacional de Doctorado de la Fundaci\'on La
Caixa-Severo Ochoa 2016 and partially suported by grant MTM2015-65825-P. .}

\begin{abstract}
Let $E$ be an infinite-dimensional separable Hilbert space. We show that for every $C^1$ function $f:E\to\mathbb{R}^d$, every open set $U$ with $C_f:=\{x\in E:\,Df(x)\; \text{is not surjective}\}\subset U$ and every continuous function $\varepsilon:E\to (0,\infty)$ there exists a $C^1$ mapping $\varphi:E\to\mathbb{R}^d$ such that $||f(x)-\varphi(x)||\leq \varepsilon(x)$ for every $x\in E$, $f=\varphi$ outside $U$ and $\varphi$ has no critical points ($C_{\varphi}=\emptyset$). This result can be generalized to the case where $E=c_0$ or $E=l_p$, $1<p<\infty$. In the case $E=c_0$ it is also possible to get that $||Df(x)-D\varphi(x)||\leq\varepsilon(x)$ for every $x\in E$.
\end{abstract}

\maketitle

\section{Introduction and main results}

Our goal in this paper is to prove the following result:

\begin{thm}\label{main result}
Let $E$ be one of the classical infinite-dimensional Banach spaces  $c_0$ or $l_p$ with $1<p<\infty$. Let $f:E\to \R^d$ be a $C^1$ function and $\varepsilon:E\to(0,\infty)$ a continuous function. Take any open set $U$ containing the critical set of points of $f$, that is $C_f:=\{x\in E:\, Df(x)\;\text{is not surjective}\}$. Then there exists a $C^1$ function $\varphi:E\to\R^d$ such that,
\begin{enumerate}
\item $||f(x)-\varphi(x)||\leq\varepsilon(x)$ for all $x\in E$;
\item $f(x)=\varphi(x)$ for all $x\in E\setminus U$;
\item $D\varphi(x)$ is surjective for all $x\in E$, i.e. $\varphi$ has no critical points; and
\item in the case that $E=c_0$ we also have that $||Df(x)-D\varphi(x)||\leq \varepsilon(x)$ for all $x\in E$.
\end{enumerate}
\end{thm}

We can make $\varphi$ be of class $C^k$ inside the open set $U$, where $k$ denotes the order of smoothness of the space $l_p$, $1<p<\infty$ or $c_0$. A brief explanation of this fact can be found in Remark \ref{higher order of differentiability on U}.

This theorem is a particular case of the following two more technical results.

\begin{thm}\label{technical version_c0}
Let $E$ be an infinite-dimensional Banach space with an unconditional basis and with a $C^1$ equivalent norm $||\cdot||$ that  locally depends on finitely many coordinates. Let $f:E\to \R^d$ be a $C^1$ function and $\varepsilon:E\to(0,\infty)$ a continuous function. Take any open set $U$ such that $C_f\subset U$. Then there exists a $C^1$ function $\varphi:E\to\R^d$ such that,
\begin{enumerate}
\item $||f(x)-\varphi(x)||\leq\varepsilon(x)$ for all $x\in E$;
\item $f(x)=\varphi(x)$ for all $x\in E\setminus U$;
\item $||Df(x)-D\varphi(x)||\leq \varepsilon(x)$ for all $x\in E$; and
\item $D\varphi(x)$ is surjective for all $x\in E$.
\end{enumerate}
\end{thm}

\begin{thm}\label{technical version_lp}
Let $E$ be an infinite-dimensional Banach space with a $C^1$ strictly convex equivalent norm $||\cdot||$ and with a $1$-suppression unconditional basis $\{e_{n}\}_{n\in\N}$, that is a Schauder basis such that
for every $x=\sum_{j=1}^{\infty}x_j e_j$ and every $j_0\in\N$ we have that
$$
\left\|\sum_{j\in\N, \, j\neq j_0} x_j e_j\right\|\leq \left\|\sum_{j\in\N}x_j e_j\right\|.
$$
Let $f:E\to \R^d$ be a $C^1$ function and $\varepsilon:E\to(0,\infty)$ a continuous function. Then for every open set $U$ such that $C_f\subset U$ there exists a $C^1$ function $\varphi:E\to\R^d$ such that,
\begin{enumerate}
\item $||f(x)-\varphi(x)||\leq \varepsilon (x)$ for every $x\in U$.
\item $f(x)=\varphi(x)$ for all $x\in E\setminus U$.
\item $D\varphi(x)$ is surjective for all $x\in E$.
\end{enumerate}

\end{thm}

The case $c_0$ and $l_p$, $1<p<\infty$ in Theorem \ref{main result} follow from Theorem \ref{technical version_c0} and Theorem \ref{technical version_lp} respectively. The reader can find the details of why this is so in Remark \ref{last_remark}.

\medskip

Note that the approximating function that we build does not have any critical point, hence it is an open mapping. 

\medskip

The classical Morse-Sard theorem \cite{Morse, Sard} states that for a given $C^k$ function $f:\R^n\to\R^d$, if $k\geq \max\{n-d+1, 1\}$ then
its set of critical values is of Lebesgue measure zero in $\mathbb{R}^{d}$. This set is defined to be the image of the set of critical points, which in turn is defined as $C_f=\{x\in\R^n:\,\rank Df(x) \;\text{is not maximum}\}$.

In general if $E$ and $F$ are Banach spaces, for a differentiable mapping
$f:E\longrightarrow F$, $C_f$ stands for the set of points $x\in E$ at which the differential $Df(x)$ is not surjective, and $f(C_{f})$ is thus the set of critical values of $f$. In the case that $E$ is of infinite dimension a natural question appears: is it possible to know that $f(C_f)$ is small in any sense by just assuming enough regularity conditions on $f$? Unfortunately the answer is no because there exist $C^{\infty}$ smooth functions $f:\ell_2\to\R$ so that their set of critical  values $f(C_{f})$ contain intervals (see Kupka's counterexample \cite{Kupka}).

A weaker question that we can ask ourselves is if at least any continuous mapping can be uniformly approximated by another one with {\em small} critical set of values. Let us mention that for many applications of the Morse-Sard theorem this is sufficient.

The first result of that type was in the case of a continuous function $f:E\to\R$, where $E$ is a separable Hilbert space. Eells and McAlpin proved in \cite{EellsMcAlpin} that in such case $f$ can be uniformly
approximated by a smooth function $g$ whose set of critical values $g(C_{g})$ is of measure zero. This was so called an {\em approximate Morse-Sard result}. However in \cite{AzagraCepedello}, a much stronger result was obtained by  Azagra and  Cepedello-Boiso: every
continuous mapping from $E$, a separable Hilbert space, into $\mathbb{R}^{d}$ can be uniformly approximated by smooth mappings {\em with no
critical points}. H\'{a}jek and  Johanis \cite{HJ} established a similar result for $d=1$ in the case that $E$
is a separable Banach space which contains $c_0$ and admits a $C^k$ smooth bump function. Also in the case
$d=1$, Azagra and  Jim\'enez-Sevilla \cite{AzagraJimenez} were able to characterize the class of separable Banach spaces $E$ such that any continuous $f:X\to\R$ can be uniformly approximated by another one of class $C^1$ without any critical point, as those Banach spaces $E$ with separable dual.

Let us comment finally about the very recent paper \cite{ADG-B}. In this work, due to Azagra, Dobrowolski and the author, many of the previous results are generalized. It is proved that for the case of $E=c_0,\ell_p,L^p$, $1<p<\infty$, and $F$ a quotient of $E$, any continuous function $f:E\to F$ can be uniformly approximated by a $C^k$ smooth one with no critical points, where $k$ is denoting the order of smoothness of the space $E$ (see \cite[Theorems 1.6, 1.7]{ADG-B} for more details).

\medskip

In the present paper we consider a different approach to this problem. Suppose that our given continuous function $f:E\to\R^d$ is already of class $C^1$ and we know that its set of critical points $C_f$ is included in some open set $U$. The question is, are we able not only to uniformly approximate $f$ by another $C^1$ function $\varphi$ without critical points but also to make $\varphi$ be equal to $f$ outside $U$?

The key will be to use a $C^1$-fine approximating result for the function $f|_U:U\to\R^d$, and this is provided by the results of  \cite{MoulisApproximation,AFG-GJL}. This corresponds to Section \ref{$C^1$-fine approximation controlling the set of critical points} of the paper.

\medskip

As a matter of fact, in \cite{MoulisApproximation}, Moulis was already able to relate $C^1$-fine approximations with approximate Morse-Sard type results. She proved that for every $C^1$ function $f:E\to F$, where $E$ is an infinite-dimensional separable Hilbert space and $F$ is a separable Hilbert space, and every continuous function $\varepsilon:E\to (0,\infty)$ there exists a $C^{\infty}$ function $g:E\to F$ such that $||f(x)-g(x)||\leq \varepsilon(x)$, $||Df(x)-Dg(x)||\leq \varepsilon(x)$ for every $x\in E$ and such that $g(C_g)$ has empty interior in $F$. Obviously we strengthen this conclusion by being able to get $C_g=\emptyset$ and considering other Banach spaces, not necessarily Hilbertian. On the other hand for the Hilbert case we cannot write as the target space an infinite-dimensional Banach space as Moulis does and also we do not get the approximation in the derivatives.

\medskip

The proof of both Theorems \ref{technical version_c0} and \ref{technical version_lp} will follow these two steps:
\begin{itemize}
\item {\bf Step 1}: Firstly we construct a $C^1$ function $g:U\to\R^d$ such that $||f(x)-g(x)||\leq \varepsilon(x)/2$ and $||Df(x)-Dg(x)||\leq\varepsilon(x)$ and such that $C_g$ either is the empty set  for the case of Theorem \ref{technical version_c0}, or is locally contained in a finite union of complemented subspaces of infinite codimension in $E$ for the case of Theorem \ref{technical version_lp}. 
\item {\bf Step 2}: We extend the function $g$ to the whole space $E$ by letting it be equal to $f$ outside $U$. Because of the $C^1$-fine approximation of Step 1 this extension is still of class $C^1$ on $E$. For the case of Theorem \ref{technical version_c0} we are done. For the case of Theorem \ref{technical version_lp} we must find a $C^1$ diffeomorphism $h:E\to E\setminus C_g$ which will be the identity outside $U$ and such that $\left\{ \{x, h(x) \} : x\in E \right\}$ refines
$\mathcal{G}$ (in other words, $h$ is {\em limited by} $\mathcal{G}$), where $\mathcal{G}$ is an open cover of $E$
by open balls $B(z, \delta_z)$ chosen in such a way that if $x, y\in B(z, \delta_z)$ then
$$
\|\varphi(y)-\varphi(x)\|\leq \frac{\varepsilon(z)}{4}\leq\frac{\varepsilon(x)}{2}.
$$
The existence of such a diffeomorphism $h$ follows by a result of Section \ref{A comment about the strong $C^k$ extraction property}, which is a consequence of some results on extractibility theory from the paper \cite[Section 2]{ADG-B}. Then, the mapping $\varphi(x):=g(h(x))$ has no critical point, is equal to $f$ outside $U$ and satisfies
$\|f(x)-\varphi(x)\|\leq\varepsilon(x)$ for all $x\in E$.
\end{itemize}

Let us fix now some notations and definitions.

We call $\{e_n\}_{n\in\N}$ the unconditional basis of  $E$ and $\{e^{*}_{n}\}_{n\in\N}$ the associated biorthogonal functionals. Let also $P_n:E\to \espan\{e_1,\dots,e_n\}$ be the natural projections defined as $P_n(\sum^{\infty}_{j=1}x_je_j)=\sum^{n}_{j=1}x_je_j$ and let $K_u$ be the unconditional constant for the basis. This constant is defined to be the least number such that for every $\{\varepsilon_j\}^{n}_{j=1}\subset\{-1,+1\}$ and every $\sum^{n}_{j=1}x_je_j\in E$, 
$$||\sum^{n}_{j=1}\varepsilon_j x_j e_j||\leq K_u||\sum^{n}_{j=1}x_j e_j||.$$
Note that $||P_n||\leq K_u$ for every $n\in\N$. We shall not confuse $K_u$ with the suppression unconditional constant $K_s$, defined as the least number such that for all (equivalent finite) set $A\subset \N$, $\|P_A\|\leq K_s$,
where $P_A$ represents the projection $P_A(x)=\sum_{j\in A} x_je_j$. We have the relation $K_s\leq K_u\leq 2K_s$. Observe also that in the statement of Theorem \ref{technical version_lp} it is required that $K_s=1$.

We say that the norm $||\cdot||$ locally depends on finitely many coordinates if for every $x\in E$ there exists a natural number $l_x$, an open neighbourhood $U_x$ of $x$, some functionals $L_1,\dots,L_{l_x}\in E^*$ and a function $\gamma:\R^{l_x}\to \R$ such that
$$||y||=\gamma(L_1(y),\dots,L_{l_x}(y))$$
for every $y\in U_x$. In particular we will make use of the fact that if the norm is of class $C^1$ and we take $v\in \bigcap^{l_x}_{j=1}\Ker L_j$, then
$$D||\cdot||(y)(v)=\lim_{t\to 0}{||y+tv||-||y||\over t}=0,$$
for every $y\in U_x\setminus\{0\}$.

A function $h:E\to E$ is said to be limited by an open cover $\mathcal{G}$ provided that the set $\left\{ \{x, h(x)\} : x\in E\setminus
X \right\}$ refines $\mathcal{G}$; that is, for every $x\in E\setminus X$, we may find a $G_x\in\mathcal{G}$ such
that both $x$ and $h(x)$ are in $G_x$.

When we say that a closed set $X\subset E$ is locally contained in a finite union of complemented subspaces of infinite codimension we mean that for every $x\in X$ there exists an open neighbourhood $U_x$ of $x$ and some closed subspaces $E_1,\dots,E_{n_x}\subset E$ complemented in $E$ and of infinite codimension such that
$$X\cap U_x\subset \bigcup^{n_x}_{j=1}E_j.$$

Finally for a $C^1$ function $f:E\to \R^d$, where $f=(f^1,\dots f^d),$ we write its Fr\'echet derivative at a point $x\in E$ by $Df(x)=(Df^1(x),\dots, Df^d(x)):E\to\R^d$, where each $Df^i(x)$ is a continuous linear functional on $E$. If $f $ is $\R$-valued we sometimes simply write $f'(x)$ for its derivative.

We will also use indistinctly the symbol $||\cdot||$ to denote the norm in $E$, $E^*$ and the euclidean norm in $\R^d$.

\section{A comment about the strong $C^k$ extraction property}\label{A comment about the strong $C^k$ extraction property}

In the proof of Theorem \ref{technical version_lp} we will need the following.

\begin{proposition}\label{extractibility of a finite union of complemented subspaces}
Let $E$ be a  Banach space with a $C^k$ smooth norm. Take an open cover $\mathcal{G}$ of an open set $U$ and a closed set $X\subset U$ that is locally contained in a finite union of complemented subspaces of infinite codimension in $E$. Then there exists a $C^k$ diffeomorphism $h:E\to E\setminus X$ which is the identity outside $U$ and is limited  by $\mathcal{G}$. 

\end{proposition}

To achieve this we will use some recent results on diffeomorphic extraction of closed sets that appear in \cite[Section 2]{ADG-B}. In that paper the next definitions are introduced.
\begin{definition}
{\em A subset $X$ of Banach space $E$ has the strong $C^k$ extraction property  with respect to an open set $U$
if $X\subseteq U$, $X$ is relatively closed in $U$, and for every open set $V\subseteq U$, every subset $Y\subseteq
X$ relatively closed in $U$ there exists a $C^{k}$ diffeomorphism $\varphi$ from $U\setminus Y$ onto
$U\setminus(Y\setminus V)$ which is the identity
on $(U\setminus V)\setminus Y$. If in addition for any $\varepsilon>0$ we can ask the diffeomorphism not to move points more than $\varepsilon$ (that is, $||\varphi(x)-x||\leq\varepsilon$ for all $x$) we will say that $X$ has the $\varepsilon$-strong $C^k$ extraction property with respect to $U$.\\
We will also say that such a closed set $X$ has locally the strong (or $\varepsilon$-strong) $C^k$ extraction
property if for every point $x\in X$ there exists an open neighbourhood $U_x$ of $x$ such that $X\cap
U_x$ has the strong ($\varepsilon$-strong respectively) $C^k$ extraction property with respect
to every open set $U$ for which $X\cap{U_x}$ is a relatively closed subset of $U$.}

\end{definition}
We have the following properties.
\begin{lem}\label{Properties of C^p extractibility}
Let us suppose that $X,X_1,X_2\subset E$ have the $\varepsilon$-strong $C^k$ extraction property with respect to an open set
$U$ of $E$. Then
\begin{enumerate}
\item  For every set $Y\subseteq X$, relatively closed in $U$, $Y$ has the $\varepsilon$-strong
$C^k$ extraction property with respect to $U$;
\item For every open subset $U'\subseteq U$, $X\cap U'$ has
the $\varepsilon$-strong $C^k$ extraction property with respect to $U'$.
\item $X_1\cup X_2$ has the $\varepsilon$-strong $C^k$ extraction property with respect to $U$.
\end{enumerate}
\end{lem}
\begin{proof}
\hspace{6cm}
\begin{enumerate}
\item This follows directly from the definition.
\item See \cite[Lemma 2.22 (2)]{ADG-B}.
\item Take $Y\subseteq X_1\cup X_2$ relatively closed in $U$ and an open set $V\subseteq U$. We want to find a $C^k$ diffeomorphism $\varphi$ from $U\setminus Y$ onto $U\setminus (Y\setminus V)$ which is the identity on $(U\setminus V)\setminus Y$ and does not move points more than $\varepsilon$.\\
Define the sets $Y_1=Y\cap X_1$ and $Y_2=Y\cap X_2$, which are relatively closed in $U$ and satisfy $Y_1\cup Y_2=Y$. In particular by (1) they have the $\varepsilon$-strong $C^k$ extraction property with respect to $U$.

\begin{enumerate}
\item [(a)] There exists a $C^k$ diffeomorphism $\varphi_1:U\setminus Y_1\to U\setminus (Y_1\setminus V)$ which is the identity on $(U\setminus V)\setminus Y_1$ and does not move points more than $\varepsilon/2$.
\item [(b)] For the open set $U\setminus Y_1$, using (2) we know that $Y_2\cap (U\setminus Y_1)=Y_2\setminus Y_1$ has the $\varepsilon$-strong $C^k$ extraction property with respect to $U\setminus Y_1$. Hence there exists a $C^k$ diffeomorphism $\varphi_2:U\setminus(Y_1\cup Y_2)\to (U\setminus Y_1)\setminus ((Y_2\setminus Y_1)\setminus V)$, which is the identity on $((U\setminus Y_1)\setminus V)\setminus (Y_2\setminus Y_1)$ and does not move points more than $\varepsilon/2$.
\end{enumerate}
Observe that
\begin{align*}
\varphi_1((U\setminus Y_1)\setminus ((Y_2\setminus Y_1)\setminus V))&=\left[ U\setminus (Y_1\setminus V)\right] \setminus \left[ \varphi_1((Y_2\setminus Y_1)\setminus V)\right] =\\
&=\left[ U\setminus (Y_1\setminus V)\right]\setminus \left[ (Y_2\setminus Y_1)\setminus V\right] =  \\
&=U\setminus (Y_1\cup Y_2)\setminus V).
\end{align*}
Hence we can define a $C^k$ diffeomorphism
$$\varphi:=\varphi_1\circ\varphi_2:U\setminus (Y_1\cup Y_2)\to U\setminus ((Y_1\cup Y_2)\setminus V,$$
which is the identity on $(U\setminus V)\setminus (Y_1\cup Y_2)$ and does not move points more than $\varepsilon$.

\end{enumerate}
\end{proof}
For this kind of sets the following abstract extractibility result holds.
\begin{thm}{\cite[Theorem 2.24]{ADG-B}}\label{abstract extractility result}
Let $E$ be a Banach space and $X$ be a closed subset of $E$ which has locally the $\varepsilon$-strong $C^k$
extraction property. Let $U$ be an open subset of $E$ and $\mathcal{G}=\left\lbrace G_r\right\rbrace _{r\in\Omega}$
be an open cover of $E$. Then there exists a $C^k$ diffeomorphism $g$ from  $E\setminus
(X\setminus U)$ onto $E\setminus X$ which is the identity on $(E\setminus U)\setminus X$ and is limited by $\mathcal{G}$.
\end{thm}

\begin{proof}[Proof of Proposition \ref{extractibility of a finite union of complemented subspaces}]
For every $x\in X$ there exists an open neighbourhood $U_x$ of $x$ and some closed subspaces $E_1,\dots,E_{n_x}\subset E$ complemented in $E$ and of infinite codimension such that
$$X\cap U_x\subseteq \bigcup^{n_x}_{j=1}E_j.$$

If $E$ admits an equivalent $C^k$ smooth norm it is known (see for instance \cite[Theorem 1.4]{ADG-B}) that given a complemented subspace $H\subset E$ of infinite codimension and the open set 
$U_x$, then $H\cap U_x$ has the $\varepsilon$-strong $C^k$ extraction property with respect to any open set $U'$ for which $H\cap U_x$ is a relatively closed subset of $U'$. Therefore thanks to Lemma \ref{Properties of C^p extractibility} $(3)$ the set $\bigcup^{n_x}_{j=1}E_j\cap U_x$ has the $\varepsilon$-strong $C^k$ extraction property with respect to any open set $U'$ for which $\bigcup^{n_x}_{j=1}E_j\cap U_x$ is relatively closed on $U'$. 

Now, using Lemma \ref{Properties of C^p extractibility} $(1)$, the set $X\cap U_x\subseteq \bigcup^{n_x}_{j=1}E_j\cap U_x$ has the $\varepsilon$-strong $C^k$ strong extraction property with respect to any open set $U'$ for which $X\cap U_x$ is relatively closed on $U'$. And this means that $X$ has locally the $\varepsilon$-strong $C^k$ strong extraction property. To conclude the proof apply Theorem \ref{abstract extractility result}, noting that we have $X\subset U$ and hence $X\setminus U=\emptyset$.

\end{proof}
For more information about diffeomorphic extraction of closed sets in Banach spaces see for instance \cite{Be,West,Renz,Renz1972,Dobro,Azagra,AzagraDobrowolski,ADG-B}.

\section{$C^1$-fine approximation controlling the set of critical points}\label{$C^1$-fine approximation controlling the set of critical points}

Let us proceed with Step 1 of the scheme of the proof of the main Theorems \ref{technical version_c0} and \ref{technical version_lp}, described in the introduction. We intend to prove the following two theorems.

\begin{thm}\label{approximation of f by another one with extractible critical set--c0}
Let $E$ be an infinite-dimensional Banach space with an unconditional basis and with a $C^1$ equivalent norm that  locally depends on finitely many coordinates. Let $U$ be an open subset of $E$, $f:U\to \R^d$ a $C^1$ function and $\varepsilon: U\to (0,\infty)$ a continuous function. Then there exists a $C^1$ function $g:U\to \R^d$ such that
\begin{enumerate}
\item $||f(x)-g(x)||\leq \varepsilon (x)$ for every $x\in U$.
\item $||Df(x)-Dg(x)||\leq\varepsilon(x)$ for every $x\in U$.
\item $C_g=\emptyset$, i.e. $g$ has no critical points.
\end{enumerate}
\end{thm}

\begin{thm}\label{approximation of f by another one with extractible critical set--lp}
Let $E$ be an infinite-dimensional Banach space with a $C^1$ strictly convex equivalent norm and with a $1$-suppression unconditional basis 
(in particular $K_u$-unconditional with $1\leq K_u\leq 2$). Let $U$ be an open subset of $E$, $f:U\to \R^d$ a $C^1$ function and $\varepsilon: U\to (0,\infty)$ a continuous function. Then there exists a $C^1$ function $g:U\to \R^d$ such that:
\begin{enumerate}
\item $|f(x)-g(x)|\leq \varepsilon (x)$ for every $x\in U$.
\item $||Df(x)-Dg(x)||\leq\varepsilon(x)$ for every $x\in U$.
\item $C_g$ is locally contained in a finite union of complemented subspaces of infinite codimension in $E$.
\end{enumerate}
\end{thm}

The proofs of these results appear in Subsections \ref{3.1} and \ref{3.2} respectively, following the ideas of the papers \cite{MoulisApproximation,AFG-GJL}.

\medskip

However, we must previously introduce an important result that is an easier and slightly different version of \cite[Lemma 5]{AFG-GJL}. The proof will mainly be the same but here we want also to study the structure of the critical set of points of the approximating function and we do not care if the approximating function has more regularity than the initial function. If the given function is 
$C^1$, it is enough for the approximating function to be 
$C^1$ as well.

For the readers convenience we present a self-contained proof, even though the arguments are the same as in \cite{MoulisApproximation,AFG-GJL}.

\begin{lem}\label{lema 5 from Az et al}
Let $E$ and $F$ be a Banach spaces. Suppose that $E$ is infinite-dimensional and has a $K_u$-unconditional basis and a $C^1$ equivalent norm. Take an open set $U$ of $E$. For every open ball $B_0=B(z_0,r_0)$ with $B(z_0,2r_0)\subseteq U$, and for every $C^1$ function $f_1:U\to F$ and numbers $\varepsilon,\eta>0$ with $\sup_{x\in B(z_0,2r_0)}||Df_1(x)||<\eta$, there exists a $C^1$ function $\Psi:E\to E$ such that for $f_2:=f_1\circ \Psi$, we have
\begin{enumerate}
\item $\sup_{x\in B_0}||f_1(x)-f_2(x)||<\varepsilon$.
\item $\sup_{x\in B_0}||Df_2(x)||<(K_u)^2 8 \eta$.
\item For every $x\in E$ there exists $n_0\in\N$ and a neighbourhood $V_0$ of $x$ such that 
$$D\Psi(y)(v)=\sum^{n_0}_{n=1}\left[ a_n(y)D||\cdot||(y-P_{n-1}(y))(v-P_{n-1}(v))y_n+\xi_n(y)v_n\right]e_n $$
for every $v=\sum^{\infty}_{n=1}v_ne_n\in E$ and $y\in V_0$, where $\xi_n,a_n:V_0\to\R$ are $C^1$ functions. 
\end{enumerate}
\end{lem}
\begin{proof}
Choose $0<r<\min \{{\varepsilon\over K_u\eta},{r_0\over K_u}\}$. Let $\varphi:\R\to [0,1]$ be a $C^{\infty}$ smooth function such that $\varphi(t)=1$ if $|t|<{1\over 2}$, $\varphi(t)=0$ if $|t|>1$ and $\varphi'(\R)\subseteq [-3,0]$.

For every $n\in\N$ we define the functions $\xi_n:E\to\R$ and $\Psi:E\to E$,
\begin{align*}
\xi_n(x)&=1-\varphi\left( {||x-P_{n-1}(x)||\over r}\right),\\
\Psi(x)&=\sum^{\infty}_{n=1}\xi_n(x)x_ne_n,
\end{align*}
where $x=\sum^{\infty}_{n=1}x_ne_n\in E$. We denote by $P_{0}$ the zero operator.

\begin{fact}\label{fact about function Psi}
The mapping $\Psi:E\to \espan \{e_n:\,n\in\N\}$ is well-defined, $C^{1}$ smooth on $E$, and has the following properties:
\begin{enumerate}
\item $||\Psi'(x)||\leq (K_u)^2 8$ for all $x\in E$;
\item $||x-\Psi(x)||\leq K_ur$ for all $x\in E$;
\item $\Psi (B_0)\subseteq B(z_0,2r_0)$.
\end{enumerate}
\end{fact}
\begin{proof} 
For any $x\in E$, because $P_n(x)\to x$ and the $||P_n||$ are uniformly bounded, there exists a neighbourhood $V_0$ of $x$ and an $n_0\in\N$ such that $\xi_n(y)=0$ for all $y\in V_0$ and $n> n_0$, and so $\Psi(V_0)\subset \espan \{e_1,\dots,e_{n_0}\}$. Thus $\Psi:E\to \bigcup^{\infty}_{n=1} \espan \{e_1,\dots,e_n\}$ is a well-defined $C^{1}$ smooth map. We next compute and estimate its derivative.

We have that
$$(\xi_n(y)y_n)'=\xi'_n(y)y_n+\xi_n(y)e^{*}_{n}.$$

If $v\in E$ and $y\in V_0$
\begin{align*}
\xi'_n(y)(v)&=-\varphi'\left( {||y-P_{n-1}(y)||\over r}\right)D||\cdot||(y-P_{n-1}(y))(v-P_{n-1}(v))r^{-1}=\\
&=a_n(y)D||\cdot||(y-P_{n-1}(y))(v-P_{n-1}(v)),
\end{align*}
where $a_n:E\to \R$ are $C^{1}$ functions, defined by $a_n(y)=-\varphi'\left( {||y-P_{n-1}(y)||\over r}\right)r^{-1}$.

Looking at the expression of $\Psi$ we compute its derivative for every $y\in V_0$,
\begin{align*}
D\Psi(y)(v)&=\sum^{n_0}_{n=1}\left[ \xi'_n(y)(v)y_n+\xi_n(y)v_n\right] e_n=\\
&=\sum^{n_0}_{n=1}\left[ a_n(y)D||\cdot||(y-P_{n-1}(y))(v-P_{n-1}(v))y_n + \xi_n(y)v_n\right] e_n.
\end{align*}

Observe that we have proved $(3)$ of Lemma \ref{lema 5 from Az et al}.

Now since $|\varphi'(t)|\leq 3$, $||(I-P_{n-1})'(y)||\leq 1+K_u$ and the derivative of the norm always has norm one, for all $y$ and all $n$ we get that
$$||\xi'_n(y)||\leq \left| \varphi'\left( {||y-P_{n-1}(y)||\over r}\right)  \right|r^{-1}||(I-P_{n-1})'(y)||\leq 3(1+K_u)r^{-1}. $$

For a fixed $x$, define $n_1=n_1(x)$ to be the smallest integer with $||x-P_{n_1-1}(x)||\leq r$. Then for any $m<n_1$, $\xi_m(x)=1$ and $\xi'_m(x)=0$, and so, for every $v\in B(0,1)$, 
\begin{align*}
||D\Psi(x)(v)||&\leq ||\sum^{\infty}_{n=n_1}\xi'_n(x)(v)x_ne_n||+||\sum^{\infty}_{n=1}\xi_n(x)v_ne_n||\leq\\
&\leq K_u\sup_{n_1\leq n}|\xi'_n(x)(v)| \,||\sum^{\infty}_{n=n_1}x_ne_n||+K_u\sup_{n}|\xi_n(x)|\,||\sum^{\infty}_{n=1}v_ne_n||\leq \\
&\leq 3K_u(1+K_u)r^{-1}||\sum^{\infty}_{n=n_1}x_ne_n||+K_u\leq 4K_u+3(K_u)^2 <8(K_u)^2,
\end{align*}
proving $(1)$.

We next estimate $||x-\Psi(x)||$.
$$||x-\Psi(x)||= ||\sum_{n\geq n_1}x_n(1-\xi_n(x))e_n||\leq K_u\sup_n |1-\xi_n(x)|\,||\sum_{n\geq n_1}x_ne_n||\leq K_ur\leq r_0,
$$
which proves $(2)$. Lastly, property $(3)$ is immediate from $(2)$ and the choice of $r$.
\end{proof}

Going back to the proof of Lemma \ref{lema 5 from Az et al} define
$$f_2(x):=f_1(\Psi(x)),$$
which is a $C^1$ function.
Firstly we have that for every $x\in B_0$,
$$||f_1(x)-f_2(x)||\leq \eta ||x-\Psi(x)||\leq \eta K_ur<\varepsilon,$$
using the Lipschitzness of $f_1$ in $B(z_0,2r_0)$.

Secondly for every $x\in B_0$,
$$||Df_2(x)||\leq ||Df_1(\Psi(x))||\,||D\Psi(x)||\leq \eta (K_u)^2 8 . $$
The proof of the Lemma is now complete.

\end{proof}

\subsection{Proof of Theorem \ref{approximation of f by another one with extractible critical set--c0}}\label{3.1} 

\begin{proof}[Proof of Theorem \ref{approximation of f by another one with extractible critical set--c0}]

Using the openness of $U$, the continuity of $\varepsilon$ and $f'$, the separability of $E$ and the assumption that the norm $||\cdot||$ locally depends on finitely many coordinates, we find a covering
$$\bigcup_{j=1}B(x^j,r_j)=U$$
of $U$ such that
\begin{enumerate}
\item [(i)] $B(x^j,4r_j)\subset U$  with $r_j\leq 1$ for every $j\in\N$.
\item [(ii)]$\varepsilon(x)\geq {\varepsilon(x^j)\over 2}$ for every $x\in B(x^j,2r_j)$.
\item [(iii)]$||Df(x)-Df(x^j)||\leq {\varepsilon(x^j)\over (K_u)^2 72}$ for every $x\in B(x^j,4r_j)$. 
\item  [(iv)] For every $j\in\N$ there exist  a number $l_j\in\N$, some linear functionals $L_{j(1)},\dots,L_{j(l_j)}$, and a $C^{1}$ function $\gamma_j:\R^{l_j}\to\R$ such that 

$$||y||=\gamma_j(L_{j(1)}(y),\dots,L_{j(l_j)}(y))$$
for every $y\in B(x^j,2r_j)$.

\end{enumerate}

Now for every $j\in \N$ choose functions $\varphi_j\in C^{1}(E;[0,1])$ with bounded derivative so that $\varphi_j(x)=1$ for $x\in B(x^j,r_j)$ and $\varphi_j(x)=0$ for $x\notin B(x^j,2r_j)$. We precisely take $\varphi_j(x)=\theta_j(||x-x^j||)$ where $\theta_j:\R\to [0,1]$ is $C^\infty$ and $\theta^{-1}_{j}(1)=(-\infty,r_j]$ and $\theta^{-1}_{j}(0)=[2r_j,\infty)$. It must be noted here that despite the fact that the norm $||\cdot||$ is not differentiable at the origin, the functions $\varphi_j$ are $C^1$ for every $x\in E$ because in a neighbourhood of $x^j$ they are constantly one.

We introduce the following constants,  
\begin{align*}
\tilde{M}_k&=\sup_{x\in B(x_k,2r_k)}||\varphi'_k(x)||, \\
M_j&=\max\{1,\sum^{j}_{k=1}\tilde{M}_k\}.
\end{align*}

Next define for every $j\in\N$,
$$h_j=\varphi_j\prod_{k<j}(1-\varphi_k).$$
One can easily check that we have the following properties:
\begin{itemize}
\item For every $x\in U$ there exists $n_x=\min \{m\in \N:\,x\in B(x_m,r_m)\}$ such that $1-\varphi_{n_x}(x)=0$ and hence $h_m(y)=0$ for every $m>n_x$ and $y\in B(x_{n_x},r_{n_x})$.
\item $\sum^{\infty}_{j=1}h_j(x)=1$ for every $x\in U$.
\item $||h'_j(x)||\leq M_j$ for every $j\in\N$ and $x\in B(x^j,2r_j)$.
\end{itemize}
In particular $\{h_j\}_{j\in\N}$ is a $C^1$ partition of unity which is subordinate to $\{B(x^j,2r_j)\}_{j\in\N}$.
\medskip

For every $j\in\N$ we apply the previous Lemma \ref{lema 5 from Az et al} for each ball $B(x^j,2r_j)$, the function $f_1(x)=f(x^j)+ Df(x^j)(x-x^j)-f(x)$ and the constants ${\varepsilon(x^j)\over 2^{j+3}M_j}$ and ${\varepsilon(x^j)\over  (K_u)^2 72}$ for $\varepsilon$ and $\eta$ respectively. Note that we can apply the Lemma \ref{lema 5 from Az et al} because
$$\sup_{x\in B(x^j,4r_j)}||Df_1(x)||=\sup_{x\in B(x^j,4r_j)}||Df(x^j)-Df(x)||\leq {\varepsilon(x^j)\over (K_u)^2 72}.$$

The resulting functions from the proof of the lemma will be called  $\delta_j=f_1\circ \Psi_j$.  In particular we have
\begin{equation}\label{estimation from Lemma 3.3}
||f(x^j)+Df(x^j)(x-x^j)-\delta_j(x)-f(x)||\leq {\varepsilon(x^j)\over 2^{j+3}M_j}
\end{equation}
and
\begin{equation}\label{norm of delta_j}
||D\delta_j(x)||\leq 8{\varepsilon(x^j)\over 72}.
\end{equation}
for every $x\in B(x^j,2r_j)$.

\medskip

Let us define finally
\begin{equation}\label{approximating function-c0}
g(x):=\sum^{\infty}_{j=1}h_j(x)(f(x^j)+Df(x^j)(x-x^j)-\delta_j(x)+T_j(x-x^j)),
\end{equation}
where $T_j:E\to\R^d$ is a continuous linear surjective operator which we next construct. Define $T_j=(T^{1}_{j},\dots,T^{d}_{j})$ inductively such that for each $i=1,\dots,d$, $T^{i}_{j}$ is a non-null element of $E^{*}$ satisfying that
$$
 T^{i}_{j}\notin \espan \{e^{*}_{n},Df^k(x^n),L_{n(1)},\dots,L_{n(l_n)},T^{k}_{1},\dots,T^{k}_{j-1},T^{1}_{j},\dots,T^{i-1}_{j}:\, n\in \N,\,1\leq k\leq d\}
$$
(note that it is the span, not the closed span); which can never fill the whole space $E^*$ because Banach spaces of infinite dimension can not have a countable Hamel basis. We also impose that their norms are small enough, more precisely,
\begin{equation}\label{norm of T_j}
||T_j||\leq\varepsilon(x^j)M^{-1}_{j}2^{-j-4}\leq{\varepsilon(x^j)\over 8}.
\end{equation}
An important property that derives from this definition of $T_j$ is that the set $\{T^{1}_{j},\dots, T^{d}_{j}\}$ is linearly independent and hence $T_j:E\to \R^d$ is a surjective linear operator. We also have that 
$$T^{i}_{j}\notin \espan \{e^{*}_{n},Df^k(x^n),L_{n(1)},\dots,L_{n(l_n)},T^{k}_{1},\dots,T^{k}_{j-1},T^{p}_{j}:\, n\in \N,\,1\leq k\leq d,\,1\leq p\leq d,p\neq i\}.$$

\medskip

Using the expression \eqref{approximating function-c0} let us  check that properties $(1),(2)$ and $(3)$ of the statement of the main theorem are satisfied for this choice of $T^{i}_{j}$.

Firstly if $h_j(x)\neq 0$, then $x\in B(x^j,2r_j)$ and
\begin{align*}
||f(x^j)+Df(x^j)(x-x^j)-&\delta_j(x)+T_j(x-x^j)-f(x)||\leq\\
&\leq ||f(x^j)+Df(x^j)(x-x^j)-\delta_j(x)-f(x)||+||T_j(x-x^j)||)\leq\\
&\leq {\varepsilon(x^j)\over 2^{j+3}M_j}+{\varepsilon(x^j)2r_j\over 8} \leq {\varepsilon(x^j)\over 2}\leq\varepsilon(x). 
\end{align*}

Therefore for every $x\in U$,
$$||g(x)-f(x)||=||\sum^{\infty}_{j=1}h_j(x)(f(x^j)+Df(x^j)(x-x^j)-\delta_j(x)+T_j(x-x^j)-f(x))||\leq   \varepsilon(x)\sum^{n}_{j=1}h_j(x)=\varepsilon(x).
$$
We have proved $(1)$.

In order to show $(2)$ and $3)$, let us analyze what the derivative of $g$ looks like, and inspect its critical set.
\begin{claim}\label{good neighbourhood for x-c0}
For every $x\in U$ there exist $n,k_1,\dots,k_n\in\N$ and a neighbourhood $V_x=V\subset B(x^n,r_n)$ of $x$ such that:
\begin{enumerate}
\item [(i)] For every $y\in B(x^n,r_n)$,
\begin{equation}\label{formula for g-c0}
g(y):=\sum^{n}_{j=1}h_j(y)(f(x^j)+Df(x^j)(y-x^j)-\delta_j(y)+T_j(y-x^j)),\;\;\text{and}
\end{equation}
\begin{equation}\label{formula for g'-c0}
Dg(y)=\sum^{n}_{j=1}h'_j(y)\left[ f(x^j)+Df(x^j)(y-x^j)-\delta_j(y)+T_j(y-x^j)\right] +\sum^{n}_{j=1}h_j(y)\left[  Df(x^j)-D\delta_j(y)+T_j\right].
\end{equation}

\item [(ii)] For every $y\in V$ and $1\leq j\leq n$,  $D\delta_j(y)(v)=Df(\Psi_j(y))\circ ( D\Psi_j(y)(v))$ has the form
\begin{equation}\label{expresssion for the derivative of delta-c0}
Df(\Psi_j(y))\circ \left\lbrace\sum^{k_j}_{n=1}\left[ a^{j}_{n}(y)D||\cdot||(y-P_{n-1}(y))(v-P_{n-1}(v))y_n+\xi^{j}_{n}(y)v_n\right]e_n\right\rbrace .
\end{equation}
\end{enumerate}
\end{claim}

\begin{proof}
Recall that for every $x\in U$ there is $n_{x}=n=\min \{m\in \N:\,x\in B(x_m,r_m)\}$ such that $h_m(y)=0$ for every $m>n$ and every $y\in B(x_n,r_n)$. So expression \eqref{approximating function-c0} becomes
$$g(y):=\sum^{n}_{j=1}h_j(y)(f(x^j)+ Df(x^j)(y-x^j)-\delta_j(y)+T_j(y-x^j))$$
for all $y\in B(x_n,r_n)$. Computing the derivative we get 
$$Dg(y)=\sum^{n}_{j=1}h'_j(y)\left[ f(x^j)+Df(x^j)(y-x^j)-\delta_j(y)+T_j(y-x^j)\right] +\sum^{n}_{j=1}h_j(y)\left[ Df(x^j)-D\delta_j(y)+T_j\right], $$
for every $y\in B(x_n,r_n)$.

For every $j=1,\dots,n$, by $(3)$ of Lemma \ref{lema 5 from Az et al}, we can find a neighbourhood $V_{x,j}\subset B(x_n,r_n)$ of $x$ and a number $k_j$ such that such that for every $y\in V_{x,j}$, 
\begin{equation*}
\begin{array}{l}
D\delta_j(y)(v)=Df(\Psi_j(y))\circ ( D\Psi_j(y)(v))= \\
Df(\Psi_j(y))\circ \left\lbrace\sum^{k_j}_{n=1}\left[ a^{j}_{n}(y)D||\cdot||(y-P_{n-1}(y))(v-P_{n-1}(v))y_n+\xi^{j}_{n}(y)v_n\right]e_n\right\rbrace .
\end{array}
\end{equation*}
Define then $V_x:=\bigcap^{n}_{j=1}V_{x,j}\subset B(x_n,r_n).$
\end{proof}

Using equation \eqref{formula for g'-c0} of Claim \ref{good neighbourhood for x-c0}, we can write 
\begin{align*}
||Dg(x)-Df(x)||\leq&||\sum^{n}_{j=1}h'_j(x)(f(x^j)+Df(x^j)(x-x^j)-\delta_j(x)+T_j(x-x^j)-f(x)||+\\
&+||\sum^{n}_{j=1}h_j(x)(Df(x^j)-D\delta_j(x)+T_j-Df(x))||\leq \\
\leq & \sum^{n}_{j=1}||h'_j(x)||\,(||f(x^j)+Df(x^j)(x-x^j)-\delta_j(x)-f(x)||+||T_j(x-x^j)||)+\\
&+\sum^{n}_{j=1}h_j(x)\left( ||Df(x^j)-Df(x)||+||D\delta_j(x)||+||T_j||\right)
\end{align*}
for every $x\in U$. Let us try to estimate all these quantities. Applying inequality \eqref{estimation from Lemma 3.3} and the bound of $||T_j||$ given by \eqref{norm of T_j} we get
$$||f(x^j)+Df(x^j)(x-x^j)-\delta_j(x)-f(x)||+||T_j(x-x^j)||\leq {\varepsilon(x^j)\over 2^{j+3}M_j}+{\varepsilon(x^j)2r_j\over 2^{j+4}M_j} $$
for every $x\in B(x^j,2r_j)$.
On the other hand  $||Df(x^j)-Df(x)||\leq {\varepsilon(x^j)\over (K_u)^2 72}\leq {\varepsilon(x^j)\over 72}$ by our choice of the partition of unity, and using \eqref{norm of delta_j} and again \eqref{norm of T_j} we have that for every $x\in B(x^j,2r_j)$,
$$||Df(x^j)-Df(x)||+||D\delta_j(x)||+||T_j||\leq {\varepsilon(x^j)\over 72}+ 8{\varepsilon(x^j)\over 72}+{\varepsilon(x^j)\over 8}={\varepsilon(x^j)\over 4}.$$
We also know that the norm of $h_j'(x)$ is bounded by $M_j$ as was indicated when stating the properties of the partition of unity. This fact together with these previous computations allow us to conclude that
\begin{align*}
||Dg(x)-Df(x)|| \leq & \sum^{n}_{j=1}M_j\left( {\varepsilon(x^j)\over 2^{j+3}M_j}+{\varepsilon(x^j)2r_j\over 2^{j+4}M_j}\right) +\sum^{n}
_{j=1}h_j(x)\left( {\varepsilon(x^j)\over 4}\right) \leq \\
\leq & \sum^{n}_{j=1}{\varepsilon(x^j)\over 2^{j+2}}+\sum^{n}_{j=1}h_j(x) {\varepsilon(x^j)\over 4} \leq {\varepsilon(x)\over 2}+{\varepsilon(x)\over 2}=\varepsilon(x)
\end{align*}
for every $x\in U$. We have then proved $(2)$ of Theorem \ref{approximation of f by another one with extractible critical set--c0}.

\medskip

Let us focus now on studying the critical set of points of $g$.

Use Claim \ref{good neighbourhood for x-c0} to choose a vector $x\in U$ for which there exist numbers $n,k_{1},\dots,k_n$ and a neighbourhood $V=V_{x}\subset B(x_n,r_n)$ such that $(i)$ and $(ii)$ of the claim hold.
Define also 
$$\tilde{n}:=\max \{n,k_1,\dots,k_n\}.$$ 

Take $(t_1,\dots,t_d)\in\R^d$ and $y\in V$. Our goal is to find a vector $v\in E$ such that $Dg(y)(v)=(t_1,\dots,t_d)$. Once we prove this we will get $(3)$ of Theorem \ref{approximation of f by another one with extractible critical set--c0}.

\medskip

With $y\in V$ fixed, looking at the formula \eqref{expresssion for the derivative of delta-c0} of Claim \ref{good neighbourhood for x-c0}, we are interested in the expression of the bounded linear operators $h'_j(y),D\delta_j(y),Df(x^j),T_j$ for $j=1,\dots,n$. Let $m=m_y$ be the least number such that $y\in B(x^m,r_m)$, that is $h_m(y)=1$ (observe that necessarily $m\leq n$), then we write equation \eqref{formula for g'-c0} as

$$Dg(y)=\sum^{m}_{j=1}h'_j(y)\left[ f(x^j)+Df(x^j)(y-x^j)-\delta_j(y)+T_j(y-x^j)\right] +\sum^{m}_{j=1}h_j(y)\left[ Df(x^j)-D\delta_j(y)+T_j\right]. $$

We want to find a vector $v\in E$ for which
\begin{align*}
\left\{
        \begin{array}{l}
         h'_j(y)(v)=0\;\;\text{ for every}\;\; 1\leq j\leq m, \\
         D\delta_j(y)(v)=(0,\dots,0);\;\text{ for every}\;\; 1\leq j\leq m,\\
         Df(x^j)(v)=(Df^1(x^j)(v),\dots,Df^d(x^j)(v))=(0,\dots,0),\;\;\text{ for every}\;\; 1\leq j\leq m,\\
         T_{j}(v)=(0,\dots,0)\;\;\text{for every}\;\;1\leq j<m,\\
         h_m(y)T_m(v)=T_m(v)=(t_1,\dots,t_d). 
           \end{array}
         \right.
\end{align*}
Let us pay attention to the vectors $y-x^j$ and $y-P_{i-1}(y)$, for $1\leq j\leq m$ and $1\leq i\leq \tilde{n}$. For simplicity let us rename these vectors as $\{z^1,\dots,z^{k_0}\}$. Each of these elements $z^k$, $1\leq k\leq k_0$, belongs to some ball $B(x^{k'},2r_{k'})$ (for each $k$ we associate a unique $k'$, not necessarily equal to $k$). So by using property (iv) from the beginning of the proof there exists a finite number of continuous linear functionals $\{L_{k'(1)},\dots,L_{k'(l_{k'})}\}$ and a $C^1$ function $\gamma_{k'}:\R^{l_{k'}}\to \R$ such that
$$||z^k||=\gamma_{k'}(L_{k'(1)}(y),\dots,L_{k'(l_{k'})}(y)).$$
We intend to take a vector $v\in \bigcap^{l_{k'}}_{j=1} \Ker\,L_{k'(j)}$, so that $D||\cdot||(z^k)(v)=0$ for every $k=1,\dots,k_0$.

For every $i=1,\dots,d$, let us introduce the finite set  of functionals
\begin{align*}
A_{i}:= &\{e^{*}_1,\dots,e^{*}_{\tilde{n}}\}\cup\{Df^j(x^1),\dots, Df^j(x^{m}):\,1\leq j\leq d\}\cup\{L_{k'(1)},\dots,L_{k'(l_{k'})}:\,1\leq k\leq k_0\}\cup\\
&\cup\{T^{j}_{1},\dots,T^{j}_{m-1}:\,1\leq j\leq d\}\cup\{T^{j}_{m}:\,1\leq j\leq d,j\neq i\}.
\end{align*}

By the definition of $T^{i}_{m}$ we have that $T^{i}_{m}\notin \espan\left( A_{i}\right)$, which is equivalent to saying that $\bigcap_{a^*\in A_{i}}\Ker\,a^*\varsubsetneq \Ker\, T^{i}_{m}. $
Therefore there exists an element $w^{i}\in E$ such that $T^{i}_m(w^{i})\neq 0
$ and $a^*(w^i)=0$ for every $a^*\in A_{i}$.

For every $i=1,\dots,d$, take $v^{i}={t_{i}w^{i}\over T^{i}_{m}(w^{i})}$ and define $v:=v^1+\cdots +v^{d}$, so we have 
$$T_m(v)=(T^{1}_{m}(v),\dots T^{d}_{m}(v))=(T^{1}_{m}(v^1),\dots T^{d}_{m}(v^d))=(t_1,\dots,t_d).$$
Moreover, $D||\cdot||(y-x^j)(v)=0$ for every $1\leq j\leq m$, $D||\cdot||(y-P_{i-1}(y))(v)=0$
for every $1\leq i\leq \tilde{n}$, and $Df(x^j)(v)=(Df^1(x^j)(v),\dots,Df^d(x^j)(v))=(0,\dots,0)$ for every $1\leq j\leq m$.
Furthermore, writing $v$ in coordinates, $v=\sum^{\infty}_{j=1}v_je_j$ we have that $v_1=\cdots=v_{\tilde{n}}=0$.

Recall that $h_j(y)=\theta_j(||y-x^j||)\prod_{k<j}(1-\theta_k(||y-x^k||))$, so 
$$h'_j(y)(\cdot)=\sum^{j}_{k=1}\gamma_{k,j}(y)D||\cdot||(y-x^k)(\cdot),$$
where $\gamma_{k,j}:E\to\R$ are $C^1$ functions. Hence with our choice of $v$ we have $h'_{j}(y)(v)=0$ for every $1\leq j\leq m$. 

On the other hand, looking at formula \eqref{expresssion for the derivative of delta-c0} of Claim \ref{good neighbourhood for x-c0}, we also get $D\delta_j(v)=0$ for every 
$1\leq j\leq m$.

Finally we also have $T_j(v)=(T^{1}_{j},\dots,T^{d}_{j}(v))=(0\dots,0)$ for every $j<m$, because $T^{1}_{j},\dots,T^{d}_{j}\in \bigcap^{d}_{i=1}A_{i}$ for every $j<m$.

Putting all these facts together, we have proved that $Dg(y)(v)=(t_1,\dots,t_d)$ and consequently the critical set of points of $g$ is empty.

\end{proof}

\subsection{Proof of Theorem \ref{approximation of f by another one with extractible critical set--lp}}\label{3.2}
\hspace{6cm}

The essence of the proof will be close to the one of the previous subsection. However there are some important changes. Here we do not rely on a norm that locally depends on finitely many coordinates, but on the property of the basis of being $1$-suppression unconditional, which will provide us with the necessary tools to approximate the function $f$ and its derivative $f'$ by another function with a {\em small} critical set of points.

\begin{proof}[Proof of Theorem \ref{approximation of f by another one with extractible critical set--lp}]
$E$ has a separable dual, so it does not contain copies of $l_1$ and since it has an unconditional basis, by \cite[Theorem 1.c.9]{LindenstraussTzafriri} we know that the basis is also shrinking, that is, $\overline{\espan} \{e^{*}_n:\,n\in\N\}=E^*$.

Using the openness of $U$, the continuity of $\varepsilon$ and $Df$, and the facts that $\overline{\espan} \{e_n:\,n\in\N\}=E$ and $\overline{\espan} \{e^{*}_n:\,n\in\N\}=E^*$, we find a covering
$$\bigcup_{j=1}B(x^j,r_j)=U$$
of $U$ and continuous linear functionals $F_j:E\to\R^d$ for every $j\in\N$ such that:
\begin{enumerate}
\item [(i)] $B(x^j,4r_j)\subset U$  with $r_j\leq 1$ for every $j\in\N$.
\item [(ii)]$\varepsilon(x)\geq {\varepsilon(x^j)\over 2}$ for all $x\in B(x^j,2r_j)$.
\item [(iii)]$||Df(x)-Df(x^j)||\leq {\varepsilon(x^j)\over (K_u)^2 144}$ for every $x\in B(x^j,4r_j)$. 
\item  [(iv)]$||F_j-Df(x^j)||\leq {\varepsilon(x^j)\over (K_u)^2 144}$.
\item [(v)] For every $j\in\N$,
\begin{align*}
\left\{
           \begin{array}{l}
         x^j=\sum^{N_j}_{i=1}\alpha_{i,j}e_i, \\
         F_j=(F^{1}_{j},\dots,F^{d}_{j})=(\sum^{N_j}_{i=1}\beta^{1}_{i,j}e^{*}_{i},\dots,\sum^{N_j}_{i=1}\beta^{d}_{i,j}e^{*}_{i}).
           \end{array}
         \right.
\end{align*}
for some $\alpha_{1,j},\dots,\alpha_{N_j,j},\beta^{q}_{1,j},\dots,\beta^{q}_{N_j,j}\in\R$, $1\leq q\leq d$, where $N_1\leq N_2\leq\dots$ is an increasing sequence of natural numbers. Note that we allow some $\alpha_{i,j}$ or $\beta^{q}_{i,j}$ to be null.
\end{enumerate}

At this point we proceed exactly as in the previous subsection, defining the $C^1$ partition of unity $\{h_j\}_{j\geq 1}$ subordinate to $\{B(x^j,2r_j)\}_{j\geq 1}$, and also the constants $\tilde{M}_k$ and $M_k$. We also apply Lemma \ref{lema 5 from Az et al}, exactly in the same way as before, but now to the function  $f_1(x)=f(x^j)+F_j(x-x^j)-f(x)$ and the constants ${\varepsilon(x^j)\over 2^{j+3}M_j}$ and ${\varepsilon(x^j)\over (K_u)^2 72}$ for $\varepsilon$ and $\eta$ respectively, obtaining $\delta_j=f\circ \Psi_j$.

\medskip

We define finally
\begin{equation}\label{approximating function}
g(x):=\sum^{\infty}_{j=1}h_j(x)(f(x^j)+F_j(x-x^j)-\delta_j(x)+T_j(x-x^j)),
\end{equation}
where $T_j:E\to\R^d$ is a continuous linear surjective operator that will be defined in the following paragraph.

\medskip 

Choose a family of pairwise disjoint subsets $\{I_n\}_{n\geq 1}$ of natural numbers such that each $I_n\subset \N$ has infinite elements and, if we denote $\mathbb{I} =\bigcup_{n\geq 1}I_n$, then $\N\setminus \mathbb{I}$ is infinite. Write also $I_n=I^{1}_{n}\cup\cdots\cup I^{d}_{n}$ as a pairwise disjoint union of sets, each of them having again infinite elements.
For every $j\in\N$ and $i=1,\dots,d$ we choose $T^{i}_{j}\in E^*$ satisfying that  
$$ T^{i}_{j}\in \overline{\espan} \{ e^{*}_{n}:\, n\in I^{i}_{j}\}\setminus \espan \{ e^{*}_{n}:\, n\in I^{i}_{j}\} .$$ 
Define $T_j:=(T^{1}_{j},\dots, T^{d}_{j})$ and also assume with no loss of generality that
$$||T_{j}||\leq\varepsilon(x^j)M^{-1}_{j}2^{-j-4}\leq {\varepsilon(x^j)\over 8}.$$

\medskip

Following the computation made for proving Theorem \ref{approximation of f by another one with extractible critical set--c0} $(1)$ in the previous subsection, we can check that for every $x\in U$,
$$||g(x)-f(x)||=||\sum^{\infty}_{j=1}h_j(x)(f(x^j)-F_j(x-x^j)-\delta_j(x)+T_j(x-x^j)-f(x))||\leq   \varepsilon(x)\sum^{n}_{j=1}h_j(x)=\varepsilon(x),
$$
which proves $(1)$.

To analyze the derivative of $g$ and its set of critical points in order to show $(2)$ and $(3)$ we also have at our disposal the following.
\begin{claim}\label{good neighbourhood for x}
For every $x\in U$ there exist $n,k_1,\dots,k_n\in\N$ and a neighbourhood $V_x=V\subset B(x^n,r_n)$ of $x$ such that:
\begin{enumerate}
\item [(i)] For every $y\in B(x^n,r_n)$,
\begin{equation}\label{formula for g}
g(y):=\sum^{n}_{j=1}h_j(y)(f(x^j)+ F_j(y-x^j)-\delta_j(y)+T_j(y-x^j)),\;\;\text{and}
\end{equation}
\begin{equation}\label{formula for g'}
Dg(y)=\sum^{n}_{j=1}h'_j(y)\left[ f(x^j)+F_j(y-x^j)-\delta_j(y)+T_j(y-x^j)\right] +\sum^{n}_{j=1}h_j(y)\left[ F_j-D\delta_j(y)+T_j\right].
\end{equation}

\item [(ii)] For every $y\in V$ and $1\leq j\leq n$,  $D\delta_j(y)(v)=Df(\Psi_j(y))\circ ( D\Psi_j(y)(v))$ has the form
\begin{equation}\label{expresssion for the derivative of delta}
Df(\Psi_j(y))\circ \left\lbrace\sum^{k_j}_{n=1}\left[ a^{j}_{n}(y)D||\cdot||(y-P_{n-1}(y))(v-P_{n-1}(v))y_n+\xi^{j}_{n}(y)v_n\right]e_n\right\rbrace .
\end{equation}
\end{enumerate}
\end{claim}

\begin{proof}
Follow the proof of Claim \ref{good neighbourhood for x-c0}.
\end{proof}

Using equation \eqref{formula for g'} of Claim \ref{good neighbourhood for x}, a straightforward calculation as in the previous subsection gives
\begin{align*}
||Dg(x)-Df(x)||\leq& \sum^{n}_{j=1}||h'_j(x)||\,(||f(x^j)+F_j(x-x^j)-\delta_j(x)-f(x)||+||T_j(x-x^j)||)+\\
&+\sum^{n}_{j=1}h_j(x)\left( ||F_j-Df(x^j)||+||Df(x^j)-Df(x)||+||D\delta_j(x)||+||T_j||\right)\leq \varepsilon(x)
\end{align*}
for every $x\in U$. We have thus proved $(2)$ of Theorem \ref{approximation of f by another one with extractible critical set--lp}.

\medskip

It remains to study the critical set of $g$.

Take a vector $x\in U$. By Claim \ref{good neighbourhood for x} there exist numbers $n,k_{1},\dots,k_n$ and a neighbourhood $V=V_{x}\subset B(x_n,r_n)$ such that (i) and (ii) of the claim hold. Define also 
$$\tilde{n}:=\max \{n,N_n,k_1,\dots,k_n\}.$$ 

Let us divide the set $\N\setminus \mathbb{I}=\mathbb{J}$ in another disjoint infinite family of subsets $\{J_n\}_{n\geq 1}$, each of them having infinite elements. Consider also the set
\begin{equation}\label{equation for A}
A=\left\lbrace y-x^j,y-P_{i-1}(y):\,j=1,\dots,n\;i=1\dots,\tilde{n}\right\rbrace ,
\end{equation}
and define $k_0:=\dim (\{\espan (A)\})\leq n+\tilde{n}$.

In order to establish $3$ of Theorem \ref{approximation of f by another one with extractible critical set--lp} our goal is to show that if
$$y\in V\setminus \left( \bigcup^{k_0}_{k=1}\overline{\espan} \{e_j:\,j=1,\dots,\tilde{n}\;\text{or}\;j\in\N\setminus J_k\}\right),$$
and $t=(t_1,\dots,t_d)\in\R$ then there exists a vector $v\in E$ such that $Dg(y)(v)=t$. Indeed for every $x\in U$ we would have found a neighbourhood $V_x=V$ such that 
$$C_g\cap V\subseteq \left( \bigcup^{k_0}_{k=1}\overline{\espan} \{e_j:\,j=1,\dots,\tilde{n}\;\text{or}\;j\in\N\setminus J_k\}\right).$$

\medskip

Fix $y\in V\setminus \left( \bigcup^{k_0}_{k=1}\overline{\espan} \{e_j:\,j=1,\dots,\tilde{n}\;\text{or}\;j\in\N\setminus J_k\}\right)$ and look at the formula of $Dg(y)$ given by property (i) of Claim \ref{good neighbourhood for x}. We are interested in the expression of the continuous linear operators $h'_j(y),F_j,D\delta_j(y),T_j$ for $j=1,\dots,n$. Let $m=m_y$ be the least number such that $y\in B(x^m,r_m)$, that is $h_m(y)=1$ (observe that necessarily $m\leq n$), then we may write equation \eqref{formula for g'} as

$$Dg(y)=\sum^{m}_{j=1}h'_j(y)\left[ f(x^j)+F_j(y-x^j)-\delta_j(y)+T_j(y-x^j)\right] +\sum^{m}_{j=1}h_j(y)\left[ F_j-D\delta_j(y)+T_j\right]. $$

We need to find a vector $v\in E$ for which

\begin{align*}
\left\{
        \begin{array}{l}
         h'_j(y)(v)=0\;\;\text{ for every}\;\; 1\leq j\leq m, \\
         D\delta_j(y)(v)=(0,\dots,0);\;\text{ for every}\;\; 1\leq j\leq m,\\
         F_j(v)=(F^{1}_{j}(v),\dots,F^{d}_{j}(v))=(0,\dots,0),\;\;\text{ for every}\;\; 1\leq j\leq m,\\
         T_{j}(v)=(0,\dots,0)\;\;\text{for every}\;\;1\leq j<m,\\
         h_m(y)T_m(v)=T_m(v)=(t_1,\dots,t_d). 
           \end{array}
         \right.
\end{align*}

By definition of $y$ there exist $j(1),\dots,j(k_0)>\tilde{n}$ such that $j(1)\in J_1,\dots,j(k_0)\in J_{k_0}$ and $y_{j(1)},\dots,y_{j(k_0)}\neq 0$. Furthermore the vectors $y-x^j$ have their $j(1),\dots,j(k_0)^{th}$-coordinates non-null because we had $x^j\in \espan\{e_1,\dots,e_{N_j}\}\subseteq\espan\{e_1,\dots,e_{N_n}\}\subseteq \espan\{e_1,\dots,e_{\tilde{n}}\}$. This implies that the $j(1),\dots,j(k_0)^{th}$-coordinates of all the vectors in the set $A$ (see expression \eqref{equation for A}) are non-null. 

We will need the following:

\begin{fact}\label{fact about suppression basis}
For every $w=\sum_{j=1}^{\infty}w_j e_j\in E\setminus\{0\}$ and every $j_0\in\N$ we have that
$$
w_{j_0}\neq 0 \implies D||\cdot||(w)(e_{j_0})\neq 0.
$$
\end{fact}
\begin{proof}
This is a consequence of the facts that the norm is strictly convex and the basis $\{e_n\}_{n\in\N}$ is $1$-suppression unconditional. For details see for example \cite[Fact 4.5]{ADG-B}.
\end{proof}

Consequently we can assure that
$$e_{j(k)}\notin \bigcap_{a\in A} \Ker (D||\cdot||(a)) $$
for every $1\leq k\leq k_0$. For every $i=1,\dots,d$, let us define $E^{i}_{(m,\tilde{n})}=\overline{\espan}\{e_n:\,n>\tilde{n}\;\text{and}\;n\in\mathbb{J}\cup I^{i}_{m}\}$. Since $k_0=codim \left( \bigcap_{a\in A}\Ker D||\cdot||(a)\right)$,  we can write 
$$E=\left( \bigcap_{a\in A}\Ker D||\cdot||(a)\right) \oplus \espan \{e_{j(1)},\dots,e_{j(k_0)}\},$$
so
$$E^{i}_{(m,\tilde{n})}=\left( \bigcap_{a\in A}\Ker D||\cdot||(a)\cap E^{i}_{(m,\tilde{n})}\right) \oplus \espan\{e_{j(1)},\dots,e_{j(k_0)}\}.$$

On the other hand $e_{j(1)},\dots,e_{j(k_0)}\in \Ker\, T^{i}_{m}$ for every $i=1,\dots,d$. In particular we can find an element 
$$w^i\in \left( \bigcap_{a\in A}\Ker D||\cdot||(a)\cap E^{i}_{(m,\tilde{n})}\right) \setminus\left( \Ker \,T^{i}_{m}\right) .$$
Otherwise we would have $\left( \bigcap_{a\in A}\Ker D||\cdot||(a)\cap E^{i}_{(m,\tilde{n})}\right) \subset \Ker\, T^{i}_{m}$ which implies that $T^{i}_{m}(w^i)=0$ for every $w\in E^{i}_{(m,\tilde{n})}$, a contradiction with the definition of $T^{i}_{m}$.

\medskip

Let us now mix all these previous ingredients together. The vector $v$ we are looking for is
$$v:=\sum^{d}_{i=1}{t_iw^i\over T^{i}_{m}(w^i)}.$$
We obviously have $T_m(v)=(T^{1}_{m}(v),\dots,T^{d}_{m}(v))=(t_1,\dots,t_d)$, so it remains to check that $h'_j(v)=0$, that $D\delta_j(v)=F_j(v)=(0,\dots,0)$ for every $j=1,\dots,m$ and that $T_j(v)=(0,\dots,0)$ for every $j<m$.

For the $h'_j$, recall that $h_j(y)=\theta_j(||y-x^j||)\prod_{k<j}(1-\theta_k(||y-x^k||))$. So we have that
$$h'_j(y)(\cdot)=\sum^{j}_{k=1}\gamma_{k,j}(y)D||\cdot||(y-x^k)(\cdot),$$
where $\gamma_{k,j}:E\to\R$ are $C^1$ functions. The elements $y-x^j$ belong to the set $A$ so it is clear that $h'_j(v)=0$ for every $1\leq j\leq m$.

For the $D\delta_j$, using \eqref{expresssion for the derivative of delta} and the facts that the elements $y-P_{i-1}(y)$ belong to the set $A$ and that the coordinates $v_1,\dots,v_{\tilde{n}}=0$, we conclude that $D\delta_j(v)=0$ for every $1\leq j\leq m$.

The fact that $F_j(v)=(0,\dots,0)$ is clear since 
$$F_j=(F^{1}_{j},\dots,F^{d}_{j})=(\sum^{N_j}_{i=1}\beta^{1}_{i,j}e^{*}_{i},\dots,\sum^{N_j}_{i=1}\beta^{d}_{i,j}e^{*}_{i}),$$
$N_j\leq N_n\leq\tilde{n}$ for every $j=1\dots,m$ and  $v_1,\dots,v_{\tilde{n}}=0$. 

Finally we also have $T_j(v)=(0,\dots,0)$ for every $j<m$, because $v\in \overline{\espan}\{e_n:n\in\mathbb{J}\cup I_m\}$ and $(\mathbb{J}\cup I_m)\cap I_j=\emptyset$ for every $j<m$.

We have proved that $Dg(y)(v)=(t_1\dots,t_d)$ and consequently the critical set of points of $g$ is locally contained in a finite union of complemented subspaces of infinite codimension in $E$. 
\end{proof}

\section{Main result}

Theorems \ref{approximation of f by another one with extractible critical set--c0} and \ref{approximation of f by another one with extractible critical set--lp} above give us an approximation of a $C^{1}$ function $f:E\to\R^d$ and of its derivative by another function $g:E\to \R^d$ which has a {\em nice} critical set of points $C_g$. In the case of Theorem \ref{approximation of f by another one with extractible critical set--c0} the term {\em nice} means we are in the best situation where $C_g=\emptyset$. And in the case of Theorem \ref{approximation of f by another one with extractible critical set--lp} the term {\em nice} will mean for us that the closed set $C_g\subseteq U$ has the $\varepsilon$-strong $C^1$ extraction property with respect to $E$, that is, there exists a $C^{1}$ diffeomorphism $h:E\to E\setminus C_g$ such that $h$ is the identity outside $U$ and $h$ refines a given open cover $\mathcal{G}$ of $E$. With these functions at our disposal, and with the help of Proposition \ref{extractibility of a finite union of complemented subspaces} we can prove our main Theorems \ref{technical version_c0} and \ref{technical version_lp}.

\begin{proof}[Proofs of Theorems \ref{technical version_c0} and \ref{technical version_lp}]
Firstly we choose another $C^1$ function $\delta:E\to[0,\infty)$ such that $\delta^{-1}(0)=E\setminus U$ and $\delta(x)\leq\varepsilon(x)$ for every $x\in E$. This is doable because in every separable Banach space with a $C^1$ equivalent norm, every closed set is the zero set of a $C^1$ function \footnote{Wells proved in his thesis \cite{Wells} that if a separable Banach space $E$ admits a $C^1$ smooth Lipschitz bump function, that is a $C^1$ non-null function $\lambda:E\to[0,\infty)$ with bounded derivative and bounded support, then every closed set $X$ of $E$ is the zero set of some $C^1$ function. Since a Banach space admitting an equivalent $C^1$ norm has a $C^{1}$ smooth Lipschitz bump function our statement is correct.}.

By Theorems \ref{approximation of f by another one with extractible critical set--c0} or \ref{approximation of f by another one with extractible critical set--lp} there exists a $C^1$ function $g:U\to\R^d$ such that
\begin{enumerate}
\item $||f(x)-g(x)||\leq {\delta (x)\over 2}$ for every $x\in U$;
\item $||Df(x)-Dg(x)||\leq{\delta(x)\over 2}$ for every $x\in U$;
\item $C_g=\emptyset$ in the case of Theorem \ref{technical version_c0}, or $C_g$ is locally contained in subspaces of infinite codimension in $E$ in the case of Theorem \ref{technical version_lp}.
\end{enumerate}
Let us extend now this function $g:U\to \R^d$ to the whole space $E$ by letting it be equal to $f$ outside $U$. We keep calling this extension by $g$ and it is important to note that this function is still of class $C^1$. The only points where this fact could not be clear are those from the boundary of $U$. However the Fr\'echet derivative of $g$ at those points $x\in\partial U$ exists and is $Df(x)$ because
\begin{align*}
\limsup_{h\to 0}{||g(x+h)-g(x)-Df(x)(h)||\over ||h||}\leq&\limsup_{h\to 0}{||g(x+h)-f(x+h)+f(x)-g(x)||\over ||h||}+\\
&+\limsup_{h\to 0}{||f(x+h)-f(x)-Df(x)(h)||\over ||h||}=\\
=&\limsup_{h\to 0}{||g(x+h)-f(x+h)||\over ||h||}+0\leq \\
\leq& \lim_{h\to 0}{\delta (x+h)-\delta(x)\over ||h||}=0.
\end{align*}

Here we are using the facts that $f$ is Fr\'echet differentiable in $\partial U$ and that $f(x)=g(x)$ and $\delta(x)=\delta'(x)=0$ for every $x\in \partial U$.

We have just shown that $g$ is Fr\'echet differentiable on $E$, but it remains to show that it is $C^1$. Straightforwardly for every $x\in\partial U$,

$$\lim_{y\to x, y\notin U}||Dg(y)-Df(x)||=\lim_{y\to x, y\notin U}||Df(y)-Df(x)||=0$$
and 
$$\limsup_{y\to x, y\in U}||Dg(y)-Df(x)||\leq\lim_{y\to x, y\in U}(||Dg(y)-Df(y)||+||Df(y)-Df(x)||)\leq\lim_{y\to x, y\in U} \delta(y)=0,$$
by the continuity of $Df$, property $(2)$ of Theorems \ref{approximation of f by another one with extractible critical set--c0} and \ref{approximation of f by another one with extractible critical set--lp} and because $\delta^{-1}(0)=E\setminus U$.

\begin{enumerate}
\item {\bf Case of Theorem \ref{technical version_c0}:}
Define $\varphi=g$ and we obtain that
$$
||\varphi(x)-f(x)||\, , \, ||D\varphi(x)-Df(x)||\leq \delta(x)\leq\varepsilon(x)
$$ for all $x\in E$ and $\varphi(x)=f(x)$ for every $x\in E\setminus U$.
Besides, it is clear that $\varphi$ does not have any critical point. 
\item {\bf Case of Theorem \ref{technical version_lp}:}
We will extract the critical set $C_g$ in the following way. Observe that $C_g$ is a closed set included in $U$ (note that $C_g\cap \partial U= \emptyset$ because $Dg(x)=Df(x)$ is surjective for every $x\in\partial U$), and by $(3)$ of Theorem \ref{approximation of f by another one with extractible critical set--lp} is locally contained in a finite union of complemented subspaces of infinite codimension. Using Proposition \ref{extractibility of a finite union of complemented subspaces}, there exists a $C^1$ diffeomorphism $h:E\to E\setminus C_{g}$
which is the identity outside $U$ and is limited by the open cover $\mathcal{G}$ that we next define. Recall that we have
$$
||f(x)-g(x)||\leq \delta(x)/2$$
for all $x\in E$. Since $g$ and $\delta$ are continuous, for every $z\in E$ there exists
$\eta_{z}>0$ so that if $x,y\in B(z,\eta_{z})$ then $||g(y)-g(x)||\leq
\delta(z)/4\leq\delta(x)/2$. We set $\mathcal{G}=\{B(x,\eta_{x})\, :\, x\in E\}$.\\
Finally, let us define $$\varphi=g\circ h.$$
Since $h$ is limited by $\mathcal{G}$ we have that, for any given $x\in E$, there exists $z\in E$ such that $x,
h(x)\in B(z,\eta_{z})$, and therefore $|g(h(x))-g(x)|\leq\delta(z)/4,$ that is, we have that
      $$
      ||g(x)-\varphi(x)||\leq\delta(z)/4\leq\delta(x)/2.
      $$
We obtain that
$$
||f(x)-\varphi(x)||\leq \delta(x)\leq\varepsilon(x)
$$ for all $x\in E$. Furthermore $h$ is the identity outside $U$ so $\varphi(x)=g(x)=f(x)$ for every $x\in E\setminus U$.
Besides, it is clear that $\varphi$ does not have any critical point: since $h(x)\notin C_{g}$, we have that the 
linear map $Dg(h(x))$ is surjective for every $x\in E$, and $Dh(x): E\to E$ is a linear
isomorphism, so $D\varphi(x)=Dg(h(x))\circ Dh(x)$ is surjective  for every
$x\in E$.
\end{enumerate}
\end{proof}

The following corollary should be compared with \cite[Theorem 1.1]{AzagraJimenez_2}, \cite[Theorem 1.5]{AzagraCepedello} or \cite[Corollary 8]{AzagraJimenez}. These results are related with the failure of Rolle's theorem in infinite-dimensional Banach spaces.
\begin{corollary}
Let $E$ be a Banach space satisfying the conditions of Theorem \ref{technical version_c0} (in particular $E=c_0$). Then for every open set $U$ there exists a $C^1$ bump function $\lambda:E\to[0,\infty)$ whose support is the closure of $U$ and does not have any critical point in $U$.

\end{corollary}

\begin{rem}\label{higher order of differentiability on U}
{\em
We could have gotten that the approximating function $\varphi$ is of class $C^{k}$ (where $k$ is the order of smoothness of the space $E$) inside the open set $U$. To achieve this one should get a version of Lemma \ref{lema 5 from Az et al} exactly as in \cite[Lemma 5]{AFG-GJL}. Doing this we would get from that lemma that the functions $\delta_j(x)$ are of class $C^{k}$. Hence the approximating function $g$ from Theorems \ref{approximation of f by another one with extractible critical set--c0} and \ref{approximation of f by another one with extractible critical set--lp},
$$g(x)=\sum^{\infty}_{j=1}h_j(x)(f(x^j)+F_j(x-x^j)-\delta_j(x)+T_j(x-x^j))$$
is a function of class $C^{k}$ on $U$. 

Moreover, we can find an extracting diffeomorphism $h:E\to E\setminus C_g$ of class $C^{k}$ by Proposition \ref{extractibility of a finite union of complemented subspaces}, hence $\varphi=g\circ h$ will be a $C^{k}$ mapping on $U$.}

\end{rem}

\begin{rem}\label{last_remark} {\em
\hspace{6cm}
\begin{enumerate}
\item 
The space $c_0$ satisfies the conditions of Theorem \ref{technical version_c0}.
The supremum norm in $c_0$ locally depends on finitely many coordinates, so applying  \cite[Theorem 1]{Haj} one gets the existence of an equivalent $C^{\infty}$ smooth norm on $c_0$ that locally depends on finitely many coordinates.
The space $C(K)$, with $K$ a metrizable countable compactum, also satisfies the conditions of Theorem \ref{technical version_c0}.

\item The space $l_p$ satisfies the conditions of Theorem \ref{technical version_lp}.
For every $1<p<\infty$ the canonical norm of $l_p$ is
$$||x||=||\sum^{\infty}_{n=1}x_ne_n||=\left( \sum^{\infty}_{n=1}|x_n|^{p}\right) ^{1/p}.$$
With this expression it is easy to check that the basis is in fact $1$-suppression unconditional with unconditional constant $K_u=1$.
It is also a norm of class $C^k$, where $k$ is defined as follows: $k=\infty$ if $p=2n$, $n\in\N$; $k=2n+1$ if $p=2n+1$, $n\in\N$, and $k$ is equal to the integer part of $p$ if $p\notin\N$. 
\end{enumerate}}
\end{rem}

\section{Acknowledgement}
The author wants to thank professor Daniel Azagra for his help and many suggestions in writing this paper.

\end{document}